\documentclass[11pt]{amsart}
\usepackage{amssymb}
\usepackage{verbatim}
\usepackage{graphicx,epsfig}

\begin{document}
\newtheorem{conjecture}{Conjecture}%[section]
\newtheorem{theorem}{Theorem}%[section]
\newtheorem{lemma}{Lemma}
\newtheorem{proposition}{Proposition}
\newtheorem{corollary}{Corollary}
\newtheorem{definition}{Definition}%[section]
\newtheorem{remark}{Remark}
\newtheorem{example}{Example}

\def\ve{\varepsilon}
\def\vp{\varphi}
\def\tk{\tilde{\kern 1 pt\topsmash k}}
%\def\barm{\bar{\kern-.2pt\bar m}}
%\def\barN{\bar{\kern-1pt\bar N}}
%\def\barA{\, \bar{\kern-3pt \bar A}}
%\def\curl{\text { curl\,}}
%\def\G{\underset\cal G \to +}
%\def\div{\text{ div\,}}

%\hsize = 6.2true in
%\vsize=8.2 true in

\def\be{\begin{equation}}
\def\ee{\end{equation}}

\def\bfx{{x}}
\def\ve{\varepsilon}
\def\iint{\not\kern -4pt\int}
\def\iiint{{\small{_\not}}\kern-3.5pt\int}
\def\mod{\text{\rm mod }}
%\def\be{\begin{equation}}
%\def\ee{\end{equation}}

% zeev's defs

\newcommand{\length}{\operatorname{length}}
\newcommand{\hd}{\length} %1-dim Hausdorff measure
\newcommand{\width}{\operatorname{width}}

\newcommand{\dist}{\operatorname{dist}}
\newcommand{\supp}{\operatorname{supp}}
\newcommand{\spec}{\operatorname{spec}}
\newcommand{\diam}{\operatorname{diam}}
\newcommand{\vol}{\operatorname{vol}}
\newcommand{\area}{\operatorname{area}}
\newcommand{\R}{\mathbb R}
\newcommand{\C}{\mathbb C}
\newcommand{\TT}{\mathbb T}
\newcommand{\Z}{\mathbb Z}
\newcommand{\HH}{\mathbb H}
\renewcommand{\^}{\widehat}
\newcommand{\vE}{\mathcal E} % lattice points on the sphere
\newcommand{\nodal}{\mathcal N} % nodal set

\numberwithin{equation}{section}

\title[On the Geometry of the Nodal Lines]
{On the Geometry of the Nodal Lines of Eigenfunctions of the
  Two-Dimensional Torus}

\author{Jean Bourgain and Ze\'ev Rudnick}
\address{School of Mathematics, Institute for Advanced Study,
Princeton, NJ 08540 }
\email{bourgain@ias.edu}

\address{Raymond and Beverly Sackler School of Mathematical Sciences,
Tel Aviv University, Tel Aviv 69978, Israel  }
\email{rudnick@post.tau.ac.il}

\date{\today}

\begin{abstract}
  The width of a convex curve in the plane is the minimal distance
  between a pair   of parallel supporting lines of the curve. In this
  paper we study   the width of   nodal lines of eigenfunctions of the
  Laplacian on the   standard flat   torus. We prove a variety of
  results on the width, some having stronger versions assuming a
  conjecture of Cilleruelo and Granville asserting a uniform bound for
  the number of lattice points on the circle lying in short arcs.

%We prove local estimates
%  on the width of ``regular'' sub-arcs of the nodal set, where regular
%  means that the curvature is pinched and the total curvature is at 
%  most one. If the eigenvalue is $\lambda^2$ then we show that the
%  width of any regular subarc of the nodal line is at most
%  $\lambda^{-1/3+\epsilon}$, and assuming a conjecture of Cilleruelo
%  and  Granville

\end{abstract}
\maketitle

\section{Introduction}

In this paper, we study the geometry of nodal lines of eigenfunctions
of the Laplacian on the standard flat torus $\TT^2=\R^2/\Z^2$.
The eigenvalues of the Laplacian on $\TT^2$ are of the form $4\pi^2 E$, where
$E=n_1^2+n_2^2$ is an integer which is a sum of two squares (in the
sequel we will abuse notation and refer to $E$ as the eigenvalue),
the corresponding eigenspace being trigonometric polynomials of the
form
%The manifold under consideration in this paper is the 2-torus $\mathbb T^2$.
%Its spectrum consists of the integers $E=\{n_1^2+n_2^2; n_i \in\mathbb
%Z\}$ and the corresponding eigenspace are the trigonometric
%polynomials of the form
\be\label{1.1}
\vp(x) =\sum_{\xi\in\mathbb Z^2, |\xi|^2 =E} a_\xi e(x \cdot \xi)
\ee
where we abbreviate $e(t):=\exp(2\pi i t)$. In order for $\varphi$
to be real-valued, the Fourier coefficients must satisfy
$\overline{a_\xi} = a_{-\xi}$.
%where $|\xi|^2 =\xi_1^2 + \xi^2_2, a_\xi =\bar a_{-\xi} \in\mathbb C$
%and $e(t)=e^{2\pi i t}$.

Given the eigenfunction $\vp$, we may consider its nodal set
\be\label{1.3}
\nodal_\vp =\{\vp=0\} \;.%\eqno{(1.3)}
\ee
According to Courant's theorem, the complement of $\nodal_\vp$ has at most
$O(E)$ connected components, the ``nodal domains''.
Their boundary are the ``nodal lines''.

For any two-dimensional surface, it is known \cite{Cheng} that the
nodal lines are a union of $C^2$-immersed circles, with at most
finitely many singular points and the nodal lines through a singular
point form an equiangular system, see Figure~\ref{plotrow}.
%When nodal lines meet, they are necessarily flat (zero geodesic
%curvature), and in that case at the intersection point they form an
%equi-angular system of curves, see Figure~\ref{plotrow}.
Thus, with the exception of the singular set, the nodal set
of an eigenfunction is rectifiable and we
can speak about its length.
In the real-analytic case, such as in our case of the flat torus,
Donnelly-Fefferman \cite{D-F} showed that the length of the nodal set
of an eigenfunction with eigenvalue $E=\lambda^2$ is commensurable to
$\lambda$:
\be\label{1.4}
\hd(\nodal_\vp)\approx \sqrt E=\lambda\;.
\ee

\begin{figure}[h]
\begin{center}
  \includegraphics[width=120mm]%[width=130mm]
{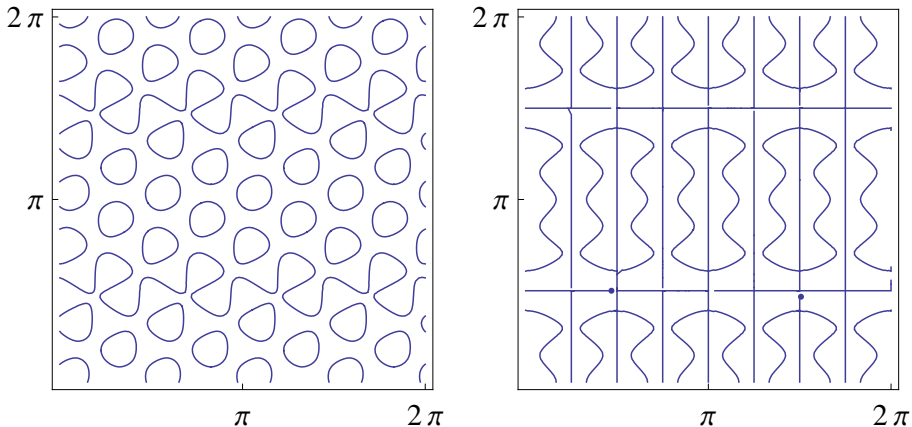}
 \caption{Nodal lines  for the eigenfunction $\cos(4 x - 7 y) + \sin(8
   x - y) + \sin(4 x + 7 y)$ (left) and
   $\sin(4x+7y)+\sin(4x-7y)+\sin(8x+y)+\sin(8x-y) = 2\sin 4x \cos y(-1
   + 2 \cos 4 x + 2 \cos 2 y - 2 \cos 4 y + 2 \cos 6 y)$
 (right).} \label{plotrow}
\end{center}
\end{figure}

Our goal in this paper is to better understand the local geometry
of nodal lines.
In this respect, M.~Berry argued \cite{Be} that for random plane
waves, the  nodal lines typically have curvature of order $E$.
If one tries to make a statement for nodal lines of individual
eigenfunctions, say in the case of $\mathbb T^2$, it is clear
that `pointwise curvature' is not the appropriate concept. 
Indeed, nodal lines (for arbitrary large $E$) may have zero curvature
or, as is also easily seen, develop arbitrary large pointwise
curvature (for a fixed $E$).

\subsection{The width of nodal lines} 
In order to formulate an alternative to curvature, we first introduce some
terminology.

\begin{definition}\label{Definition1}
An arc $C\subset\mathbb T^2$ is called `regular' if $C$ admits an
arc-length parametrization $\gamma:[0, \ell]\to C$, $0<\ell<1$,
which is $C^2$ and such that for some $\kappa>0$, the curvature
$|\ddot \gamma|$ satisfies  a pointwise pinching condition
\be\label{1.9}
\kappa <|\overset{..}\gamma|<2\kappa%\eqno{(1.9)}
\ee
and the total curvature is bounded:
\be\label{1.10}
2\kappa\ell<1 \;.% \eqno{(1.10)}
\ee
\end{definition}

%\begin{definition}\label{Definition2}
%The quantity
%\be\label{1.11}
%w(C) \approx \ell^2\kappa
%\ee
%will be called the `width' of $C$.
%\end{definition}

For a  convex curve $C$ the {\em width} $w(C)$ is defined as the
minimal distance between a pair of parallel supporting lines of the
curve (Figure~\ref{regulararc}).
\begin{figure}[h]
\begin{center}
  \includegraphics[width=70mm]%[width=130mm]
{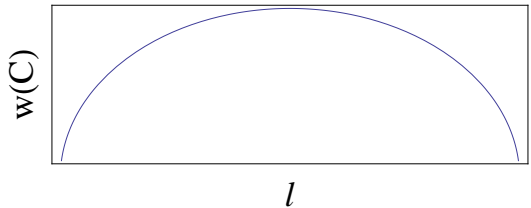}
 \caption{The width of a regular arc} \label{regulararc}
\end{center}
\end{figure}
In the case of regular arcs,  we shall see that
\be
\label{1.11}
w(C) \approx \ell^2\kappa
\ee

An examination of numerical plots of some nodal sets (see
Figure~\ref{plotrow}) leads one to realize that $\nodal_\vp$ does
not contain ``large'' curved arcs, specifically that as $E\to
\infty$, for regular arcs $C\subset \nodal_\vp$, either their length
$\ell\to 0$ shrinks or the curvature  pinching $\kappa\to 0$.
%\begin{figure}[h]
%\begin{center}
%  \includegraphics[width=70mm]%[width=130mm]
%{nodalplot.eps}
% \caption{Nodal lines  for the eigenfunction $\cos(4 x - 7 y) + \sin(8 x - y) + \sin(4 x + 7 y)$} \label{nodalplot}
%\end{center}
%\end{figure}

We make the following
conjecture:

%Examining the structure of $\nodal_\vp$ for various eigenfunctions $\vp$ of
%$\mathbb T^2$ and large $E$ (say numerically), one comes
%to the realization that $\nodal_\vp$ does not contain `large' curved arcs
%when $E\to\infty$.
%More specifically, if $C\subset \nodal_\vp$ is a regular arc as defined
%above, then either $\ell\to 0 $ or $\kappa\to 0$ for $E\to\infty$.
%We make the following conjecture.

\begin{conjecture}\label{Conjecture1}
For all $\ve>0$, there is $c_\ve>0$ such that if $\vp$ is an
eigenfunction of $\mathbb T^2$ of eigenvalue $E=\lambda^2$ and $C$ any
regular arc contained in $\nodal_\vp$, then
\be\label{1.12}
w(C)<c_\ve \lambda^{-1+\ve}.
\ee
\end{conjecture}

\noindent
This is our substitute for the phenomenon M.~Berry pointed out for
random plane waves.
The above conjecture seems to be consistent with numerics
and we will moreover prove its validity for `most' eigenvalues $E$.

In generality, we prove

\begin{theorem}\label{Theorem1}
If $C\subset \nodal_\vp$ is a regular arc, then
\be
\label{1.13}
w(C)< C_\ve \lambda^{-\frac 13+\ve}.
\ee
\end{theorem}

The argument makes crucial use of the structure of lattice points
on the circle $\{|\xi|=\lambda\}$.
Relevant results will be presented in \S~\ref{sec:lattice}.

In \S~\ref{sec:5} we will show that the exponent $1/3$ of
Theorem~\ref{Theorem1} can be improved to $1/2$ for almost all of the
nodal line, in the following sense:
\begin{theorem}\label{Theorem2}
Given $\ve>0$, there is $\delta>0$ such that the following holds.
Let $\{C_\alpha\}$ be a collection of disjoint regular arcs of $\nodal_\vp$
satisfying
\be\label{1.14}
w(C_\alpha)> \lambda^{-\frac 12+\ve} \text { for each } \alpha.
\ee
Then for $\lambda>\lambda(\epsilon)$,
\be
\label{1.15}
\sum_\alpha \length( C_\alpha) <\lambda^{1-\delta}
\ee
%(for $\lambda$ large enough, depending on $\ve$).
\end{theorem}

Recall that $\hd(\nodal_\vp)\approx\lambda$, by \eqref{1.4}, so that
Theorem~\ref{Theorem2} asserts that arcs of large width form a
negligible part of the nodal set.
\medskip

As we will see, the exponent $\frac 12$ in \eqref{1.14} could be
replaced by $1$, assuming the validity of the Cilleruelo-Granville
conjecture \cite{C-G}, stating that for all $\ve>0$, there is a
constant $B_\ve$ such that any arc on a circle $\{|\xi|=\lambda\}$
of size at most $\lambda^{1-\ve}$ contains at most $B_\ve$ lattice
points (uniformly in $\lambda$).
\medskip

Finally, we also show that Conjecture~\ref{Conjecture1} holds for at
least a positive proportion of the nodal set, in the following sense:

\begin{theorem}\label{Theorem3}
There is a constant $c_0>0$ such that the following holds.
Let $\ve>0$, $\lambda$ large enough and $\{C_\alpha\}$ a collection of
disjoint regular arcs of $\nodal_\vp$ satisfying
\be\label{1.16}
w(C_\alpha)> \lambda^{-1+\ve} \text { for each } \alpha.
\ee
Then
\be
\label{1.17}
\length(\nodal_\vp\backslash \bigcup_\alpha C_\alpha)>c_0\lambda.
\ee
\end{theorem}

The proofs of Theorems~\ref{Theorem2} and \ref{Theorem3} make
essential use of the results of Donnelly and Fefferman \cite{D-F}.

\medskip

\subsection{Total curvature} 
Another geometric characteristic of nodal lines that one can investigate is 
their {\em total curvature}.  

%We recall the notion of total curvature for a curve $C$ 
%in Euclidean space (at most $3$-dimensional ?). ??????? 
For $C^2$ curves in $\R^3$, if $\gamma:[0,\ell] \to C $ 
is a $C^2$ arc length parametrization then the total curvature is 
\begin{equation}\label{integral formula}
K(C)=\int_0^\ell ||\ddot{\gamma}(s)|| ds \;.  
\end{equation}

When one varies the curve $C$, the formula \eqref{integral formula} is
clearly continuous in the $C^2$ topology and hence can be used to
define the total curvature of any continuous curve as the limit of
the total curvature of its smooth perturbations. 
However, there is definition of total curvature which
makes sense for any continuous curve, 
which starts with defining the total curvature of a polygon as the sum of
the  angles subtended by the prolongation of 
any of its sides and the next one, 
and then for any continuous curve $C$ setting 
\be \label{polygon formula}
K(C) = \sup_{P} K(P)
\ee
where the supremum
\footnote{Alternatively one can take 
%Perhaps a more fundamental definition is to take 
\begin{equation}
  K(C) = \lim_P K(P)
\end{equation}
where the limit is over all polygons $P$ inscribed in $C$ for which
the maximal distance between adjacent vertices tends to zero; this
definition works for curves in arbitrary Riemannian manifolds \cite{CL}.}
is over all polygons $P$ inscribed in $C$. 
One can show that for $C^2$ curves this definition coincides with
\eqref{integral formula}  (see \cite{Milnor}).  
%See Milnor's beautiful undergraduate thesis
%\cite{Milnor} for background on this. 

We claim that the total curvature $K_\vp$ of the nodal set for an
eigenfunction $\vp$ with eigenvalue $E$ is bounded by 
\begin{equation}\label{total curvature bd}
  K_\vp \ll E
\end{equation}
Note that there is no lower bound, since the nodal set of the 
eigenfunction $\sin nx$ is a union of non-intersecting lines hence has
zero total curvature. % for all $n$. 

To prove \eqref{total curvature bd}, 
it suffices to assume  that nodal set is smooth,
which is easily seen to be a generic condition in the eigenspace  on
the torus, hence a small perturbation in the eigenspace will bring us
to that setting and one then invokes continuity of the total
curvature in the $C^2$-topology.  
In case the nodal set is smooth, one can make the following comment
based on the fact that $\nodal_\vp$ is a semi-algebraic set.
First observe that $\vp$ in \eqref{1.1} may be expressed as
\be
\label{1.5}
\vp(x_1, x_2)=\sum_{\alpha, \beta} a_{\alpha, \beta} (\cos
x_1)^{\alpha_1} (\cos x_2)^{\alpha_2} (\sin x_1)^{\beta_1}
(\sin x_2)^{\beta_2}%\eqno{(1.5)}
\ee
with $\alpha,\beta \in\mathbb Z_+^2$ and $\alpha_1+\alpha_2+\beta_1+\beta_2\leq \sqrt 2\lambda$.

%Introducing variables $u_1=\cos x_1, u_2= \cos x_2, v_1 =\sin x_1,
%v_2=\sin x_2$, we see that $\nodal_\vp$ is contained in the algebraic variety
%\be\label{1.6}
%\begin{cases}
%\sum a_{\alpha\beta} u_1^{\alpha_1} u_2^{\alpha_2} v_1^{\beta_1}v_2^{\beta_2}=0\\
%u_1^2+v_1^2=1\\
%u_2^2+v_2^2=1.\end{cases}%\label{(1.6)}
%\ee
%
%Eliminating the variables $v_1, v_2$ in \eqref{1.6}, it follows that
Introducing variables $u_1=\cos x_1$, $u_2= \cos x_2$, it follows that  
$u_1, u_2$ satisfy a polynomial equation
\be
\label{1.7}
P(u_1, u_2) =0%\eqno{(1.7)}
\ee
with $P\in\mathbb R[u_1, u_2]$ of degree $d<c\lambda$.
According to \cite[Theorem 4.1, Proposition 4.2]{R}, assuming
$\{P=0\}$ is smooth, its total curvature\footnote{Since the 
total curvature of an arc is the variation of its
  tangent vector, a bound is obtained by integration in $s$ of the
number of solutions in $u=(u_1, u_2)$ of
$$
\begin{cases}
\partial_2P(u)+s\partial_1 P(u)=0\\ P(u)=0.\end{cases}
$$}
is at most $const. d^2 \ll E$.
%\marginpar{assuming $\{P=0\}$ is smooth!}
Since $\nodal_\vp \subset \{P=0\}$ (in the $u_1,
u_2$-parametrization), we may conclude that \eqref{total curvature bd}
holds. 
%\be\label{1.8}
%K_\vp<cE \;.%\eqno{(1.8)}
%\ee
%This is again a global estimate on our nodal set.

\subsection{Remarks}

\begin{enumerate}
\item In defining regular arcs, one could make further higher
  derivative assumptions on the parametrization $\gamma$ (as we will
show with an example in Appendix~\ref{appendix:example}, those do not
hold automatically).
Involving higher derivatives would allow to improve upon the estimate
\eqref{1.13}.
We do not pursue this direction here however partly because Definition
\ref{Definition1} would have to be replaced by a more technical one
and it is not clear  which version would be the most natural.

\item
%Let us point out that our width-estimates are specific for $M=\mathbb T^2$.
%If for instance we take $M=S^2$ and consider spherical harmonics
%$\vp\in[Y_{\ell, m}=P_{\ell, m}(\theta) e^{im\vp}; -\ell \leq m
%\leq \ell]$, we see that circles $[P_{\ell, m} =0]$ provide families
%of regular arcs with geodesic curvature bounded away from zero.

We point out that our estimates for the width are specific to the
flat torus. For instance, they are not valid on the sphere $S^2$.
Indeed, the standard spherical harmonics $Y_{\ell,m} =
P_{\ell,m}(\theta)e^{im\varphi}$ are eigenfunctions for which the
circles of latitude $\{P_{\ell,m}(\theta)=0\}$ are families of
regular arcs with geodesic curvature bounded away from zero.

\item
%Using a similar argument, one obtains easily the analogue of
%\eqref{1.8} for the total curvature of nodal sets of
%$S^2$-eigenfunctions.
%At the time of this writing, it is not clear to the authors how to
%prove an estimate of the type \eqref{1.8} or even, more
%modestly, an explicit bound
%\be\label{1.19}
%K_\vp< k(E)
%\ee
%for $M$ a compact real-analytic manifold.

One can easily obtain the analogue of \eqref{total curvature bd} 
for the total curvature of the nodal sets on the sphere using similar
arguments to those on the torus. % (see \S~\ref{sec:total}).
At the time of this writing, it is not clear to us if there is 
an estimate of the type \eqref{total curvature bd} 
for general real-analytic surfaces, or even, more modestly,
any explicit bound $K_\varphi<K(E)$ for the total curvature.

\end{enumerate}

\bigskip

\noindent{\bf Acknowledgements:} We thank Misha Sodin for his
comments. 
J.B. was supported in part by N.S.F. grant DMS 0808042.
Z.R. was supported by the Oswald Veblen Fund during his stay at the
Institute for Advanced Study and by the Israel Science Foundation
(grant No. 1083/10).

\section{Lattice points on circles}\label{sec:lattice}

In this section, we collect some facts about lattice points on arcs
for later use.
Let
$E=\lambda^2\in\mathbb Z_+$  and
$$
\vE=\mathbb Z^2 \cap \{|x|=\lambda\}
$$
%The dimension of the eigenspace is the number
Then $|\vE|=r_2(E)$ is the number of representations of $E$ as a sum of two
squares, which is essentially the number of divisors of $E$ in the
ring of Gaussian integers. In particular one has an upper bound
%The dimension of this eigenspace %(= multiplicity of $E$)
%is the number $r_2(E)$ of representations of $E$ as a sum of two squares.
%It is well-known that $r_2(E)$ may be re-expressed in terms of
%divisors of $E$ in the Gaussian integers $\mathbb Z+i\mathbb Z$.
%In particular, there is an upper bound for $E\to\infty$
\be\label{1.2}
r_2(E) \ll \exp c \frac {\log E}{\log\log E}\ll E^\ve \text { for all }
\ve>0\;. %\eqno {(1.2)}
\ee

The next statement is a slight specification of a more general result
due to Jarnik \cite{J}.

\begin{lemma}\label{Lemma2.1}
Let $P_0, P_1, P_2\in\mathcal E$ be distinct and $|P_0- P_1|\leq|P_0-P_2|$.
Then
\be\label{2.2}
|P_0 -P_2|^2.|P_0-P_1|>c\lambda
\ee
\end{lemma}
\noindent
(here and in the sequel, $c, C$ will denote constants).
\begin{proof}
$P_0, P_1, P_2$ belong to an arc $C\subset\{|x|=\lambda\}$ of size $r$
  and we may obviously assume $r<\sqrt\lambda$.
Since $P_0, P_1, P_2$ are distinct, they span a triangle $T$ of area
$$
0<\area (T) =\frac 12\Big|\det  \begin{pmatrix} 1&P_0\\ 1&
  P_1\\ 1&P_2\end{pmatrix} \Big| \in \frac 12\mathbb Z_+.
$$
Hence, from geometric considerations
$$
\frac 12 \leq \text { area } (T) <c \frac{r^2}\lambda. |P_0-P_1|
$$
\end{proof}

\begin{lemma} \label{Lemma2.3}
Let $P_0, P_1, Q_0, Q_1 \in\mathcal E$ be  distinct  points on an arc
of size $r$. Then
\be\label{2.4}
|P_0-Q_0|.|P_1-Q_1|. r> c\lambda.
\ee
\end{lemma}

\begin{proof}

We may assume $r<\frac 1{100} \lambda$.
For $\alpha= 0, 1$, let
$$
\begin{aligned}
P_\alpha &=\lambda e^{i\theta_\alpha}\\[6pt]
P_\alpha - Q_\alpha &= \Delta_\alpha e^{i\psi_\alpha}.
\end{aligned}
$$
Then (possibly permuting $P_\alpha, Q_\alpha$)
$$
\Delta^2_\alpha =|P_\alpha -Q_\alpha|^2 =2P_\alpha.(P_\alpha - Q_\alpha)=  2\lambda\Delta_\alpha \cos
(\theta_\alpha-\psi_\alpha)
\sim 2\lambda\Delta_\alpha \Big(\frac\pi 2+ \theta_\alpha -\psi_\alpha\Big)
$$
implying that
\be\label{2.5}
\begin{aligned}
\psi_\alpha &= \frac\pi 2+\theta_\alpha+0\Big(\frac{\Delta_\alpha}\lambda\Big)\\
|\psi_0-\psi_1|& =|\theta_0-\theta_1|+ 0\Big(\frac{\Delta_0+\Delta_1}\lambda\Big)
\end{aligned}
\ee
Since the vectors $P_0-Q_0, P_1-Q_1$ are not parallel,
$$
|\det (P_0-Q_0, P_1-Q_1)|\geq 1
$$
and thus
\be\label{2.6}
\Delta_0.\Delta_1.|\psi_0-\psi_1|\geq 1
\ee

From \eqref {2.5} \eqref {2.6}
$$
1\leq \Delta_0\Delta_1 |\theta_0-\theta_1|+0\big(\lambda^{-1} \Delta_0\Delta_1(\Delta_0+\Delta_1)\big)
$$
and
$$
\lambda< 2\Delta_0\Delta_1 |P_0-P_1|+
0\big(\Delta_0\Delta_1(\Delta_0+\Delta_1)\big) < Cr\Delta_0\Delta_1 \;.
$$
\end{proof}

Let us also recall the results from Cilleruelo-Cordoba\cite {C-C} and
Cilleruelo-Granville \cite{C-G}
on the spacing properties of systems $\{P_1, \ldots, P_m\}$ of
distinct elements of $\mathcal E$.

\begin{lemma}\label{Lemma2.7} (\cite{C-C}, \cite{C-G})
$$
\prod_{1\leq i<j\leq m} |P_i-P_j|\geq\begin{cases}
\lambda^{\frac m2 (\frac m2-1)} \text { if $m$ is even}\\
\lambda^{\frac 14(m-1)^2} \text { if $m$ is odd.}
\end{cases}
$$
\end{lemma}

The argument in \cite{C-C} is arithmetic and based on factorization of
$E=\lambda^2$ in Gaussian primes.
The following elegant and much simpler
argument was given by Ramana \cite{Ramana}:
We identify the standard lattice $\Z^2\subset \R^2$ with the Gaussian integers $\Z[\sqrt{-1}]\subset \C$.
If  $\overline{ P}$ denotes the complex conjugate of $P$, then  our condition on the lattice points being on one circle says that
\begin{equation}\label{magnitude of Pj}
 P_j\overline{P_j} = \lambda^2, \quad j=1,\dots,m
\end{equation}
Ramana observed that for any $0\leq k\leq m-1$, we have an identity
\begin{equation} \label{Ramana's identity}
 \lambda^{k(k+1)}\prod_{1\leq i< j\leq m} (P_i-P_j) = \prod_{i=1}^m P_i^k \cdot \det V_{k,m}
\end{equation}
 where $ V_{k,m}$ is the following Vandermonde type matrix
\begin{equation}
 V_{k,m} =   \begin{pmatrix} \overline{P_1}^k & \overline{P_2}^k& \dots & \overline{P_m}^k\\
\overline{P_1}^{k-1}&\overline{P_2}^{k-1}& \dots & \overline{P_m}^{k-1}\\
\vdots \\
1&1&\dots &1\\
P_1&P_2&\dots& P_m \\
\vdots \\
P_1^{m-1-k}&P_2^{m-1-k}& \dots & P_m^{m-1-k}\\
                \end{pmatrix}
\end{equation}
To see this, we compute the RHS of \eqref{Ramana's identity}
by noting that $P_i^k \det V_{k,m}$ is the determinant of the matrix
resulting from multiplying the $i$-th column of $V_{k,m}$ by $P_i^k$, and using \eqref{magnitude of Pj}
one is reduced to computing an ordinary Vandermonde determinant,
yielding the LHS of \eqref{Ramana's identity}.

Once \eqref{Ramana's identity} is established,
we take absolute values and noting that  $|\det V_{k,m}|^2 \geq 1$
since it is a nonzero integer , we get
\begin{equation}
 \lambda^{k(k+1)}\prod_{1\leq i< j\leq m} |P_i-P_j| \geq \lambda^{km}
\end{equation}
Taking $k=\lfloor \frac m2 \rfloor$ gives Lemma~\ref{Lemma2.7}.
\qed
%\eqref{CC ineq}.

%\medskip

Taking $m=2$ in Lemma~\ref{Lemma2.7}, it follows that
$$
|P_0-P_1| \ |P_1-P_2| \ |P_2 -P_0|\geq \lambda
$$
and we are recovering \eqref {2.2}.
\medskip

Lemma \ref{Lemma2.7} implies a uniform bound $B(\ve)$ on the number of
elements of $\mathcal E$  on an arc $C\subset \{|x|=\lambda\}$ of size
$r<\lambda^{\frac  12-\ve}$.
More precisely
\begin{lemma}[\cite{C-C}]\label{Lemma2.8}
Let $\delta(m) = \frac 1{4\lfloor \frac m2 \rfloor +2}$.
If $C\subset \{|x|=\lambda\}$ is an arc of length
$r<\sqrt{2}\lambda^{\frac 12-\delta(m)}$,  then $\#\vE\cap C \leq m$.
\end{lemma}

%\begin{lemma}\label{Lemma2.8}  {\rm (\cite {C-C})}.
%If $r<\sqrt{ 2}\lambda^{\frac 12-(4[\frac m2]+2)^{-1}}$, then
%$\#(C\cap\mathcal E)\leq m$.
%\end{lemma}

Cilleruello and Granville conjectured a uniform bound on the number of lattice points
on any arc of length $\lambda^{1-\epsilon}$:
\begin{conjecture}\label{Conjecture2}\cite[Conjecture 14]{C-G}
Let $0<\epsilon<1$. Then there is some $B_\epsilon >0$ so that the number of lattice points
on any arc  $C\subset \{|x|=\lambda\}$ of length $r<\lambda^{1-\epsilon}$ is at most $B_\epsilon$.
\end{conjecture}

%\begin{conjecture}\label{Conjecture2}{\rm (\cite {C-G}, Conjecture 14).}
%The number of lattice points $\#(C\cap\mathcal E)$ on an arc $C\subset
%\{|x|=\lambda\}$ of length $\lambda^{1-\ve}$ is bounded
%uniformly in $\lambda$.
%\end{conjecture}

Conjecture~\ref{Conjecture2} is true for most $E=\lambda^2\in\mathbb
Z_+$, in fact we have the stronger statement that all lattice points
on the circle of  radius $\sqrt{E}$ are well separated. To make sense
of it, recall  that the number of $E\leq N$ which are a sum of two
squares is  asymptotic to a constant multiple of $N/\sqrt{\log N}$

%Conjecture \ref{Conjecture2} is certainly true for most
%$E=\lambda^2\in\mathbb Z_+$.
%Recall that
%$$\# \{E=n^2_1+n_1^2\leq N\}\sim const. \frac N{\sqrt {\log N}}$$
%In fact, one has the stronger statement

\begin{lemma}\label{Lemma2.9}
 Fix $\epsilon>0$. Then for all but $O(N^{1-\epsilon/3})$ integers
   $E\leq N$, one has
   \begin{equation}\label{2.10}
     \min_{\substack{ x\neq y\in \Z^2\\ |x|^2=|y|^2=E}} |x-y| >(\sqrt{E})^{1-\epsilon}
   \end{equation}
%Fixing $\ve>0$ and taking $N\in\mathbb Z_+$ large, $E\in \{1, 2, \ldots, N\}$ and $\lambda=\sqrt E$, one has
%\be\label{2.10}
%\min_{\substack {x\not= y\in\mathbb Z^2\\ |x|=\lambda=|y|}} |x-y|>\lambda^{1-\ve}
%\ee
%except for a subset of $E$-values of size at most $N^{1-\frac\ve3}$.
\end{lemma}

 \begin{proof}
   We will say that $E\leq N$ is ``exceptional'' if there is a pair of
   close points $|x|^2=|y|^2=E$,
   $0<|x-y|<\sqrt{E}^{1-\epsilon}$. Writing $z=x-y$, we see that
   the number of exceptional $E$'s is bounded by the number of pairs
   of integer vectors $x\in \Z^2$, $0\neq z\in \Z^2$ with
   \begin{equation}
     |x|^2 \leq N, \quad 0<|z|<\sqrt{E}^{1-\epsilon}
   \end{equation}
and satisfying
\begin{equation}\label{2xz=|z|^2}
  2x\cdot z = |z|^2
\end{equation}

Writing $z=dz'$ with $d\geq 1$  and $z'\in \Z^2$ primitive, we see
that the number of $x<\sqrt{N}$ lying on the line \eqref{2xz=|z|^2} is
$O(\sqrt{N}/|z'|)$ and hence the number of exceptional $E\leq N$ is
dominated by
\begin{equation}
  \sum_{1\leq d\leq \sqrt{N}^{1-\epsilon} } \sum_{\substack{ z'\in
      \Z^2 \mbox{ primitive}\\ |z'|\leq (\sqrt{N})^{1-\epsilon}/d}}
  \frac{\sqrt{N}}{|z'|}  \ll \sqrt{N}  \sum_{1\leq d\leq
    \sqrt{N}^{1-\epsilon} } \frac{(\sqrt{N})^{1-\epsilon}} {d} \ll
  N^{1-\epsilon/2} \log N
\end{equation}
which proves our claim.
 \end{proof}

%\noindent
%{\bf Proof of Lemma \ref{Lemma2.9}.}
%\medskip
%
%Let $M=\sqrt N$ and estimate the size of the set
%\be\label{2.11}
%S=\{x\in\mathbb Z^2; |x|\leq M\text { and } 2x.z=|z|^2 \text { for some $z\in\mathbb Z^2, 0<|z|<M^{1-\ve}$}\}.
%\ee
%Writing $z=d.z', d\in\mathbb Z_+$ and $z' =(z_1', z_2')\in\mathbb Z^2$ primitive, the equation
%\be\label{2.12}
%2x.z'=d|z'|^2
%\ee
%has at most $c\frac M{|z'|}$ solutions in $x$, $|x|\leq M$, for given $z'$ primitive.
%
%
%Hence
%$$
%|S|\leq c\sum_{d<M} \sum_{\substack{z'\in\mathbb Z^2 \\ 0<|z'|<\frac{M^{1-\ve}}d}} \frac M{|z'|
%}<c\sum_{d<M} \ \frac {M^{2-\ve}}d <cM^{2-\ve} \log M.
%$$
%Since $|S|$ is an obvious upper bound for the number of exceptional
%$E\in\{1, \ldots, N\}$, Lemma \ref{Lemma2.9} follows.

\section{The width of a regular arc}

%\subsection{The width of regular arcs}
Recall that the width of a convex curve $C$ is defined as the
minimal distance between a pair of parallel supporting lines of the
curve. We denote it by $w(C)$. 
%In the case of regular arcs,  we claim that
\begin{lemma}
Let $C$ be a regular arc, that is admitting an arc length
parametrization $\gamma:[0,\ell]\to C$ with curvature pinched by
$\kappa <|\ddot{\gamma}|<2\kappa$ and with total curvature bounded by
$2\kappa \ell <1$. Then the width of $C$ is commensurate with
\be
\label{1.11'}
w(C) \approx \ell^2\kappa
\ee
\end{lemma}
\begin{proof}
We may present $C$ as the graph of a function $f$:
$$ C = \{(x, f(x)): 0<x<L \}$$
where $f(x) \geq 0$, and $f(0)=0=f(L)$ (see Figure~\ref{widthfig}).

\begin{figure}[h]
\begin{center}
  \includegraphics[width=70mm]%[width=130mm]
{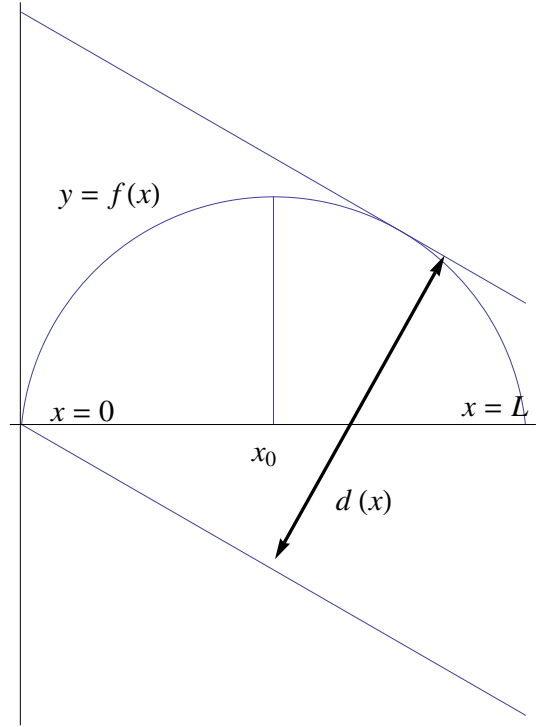}
 \caption{Computing the width of a regular arc}
\label{widthfig}
\end{center}
\end{figure}

 Note that our assumptions in particular imply that the arc is convex,
 since there are no inflection points (the curvature is nowhere zero)
 and the total curvature is small. Hence $f''<0$, and
the function $f$ has a unique
 critical point at $x_0\in (0,L)$ where $f$ is maximal.

We now note that the assumption of total curvature being at most
$1$ implies a bound for the derivative of $f$:
$$ |f'(x)| < 2 $$
Indeed,  $f'(x)=\tan \theta(x)$ where $\theta(x)$ is the angle between the
tangent vector to the arc at $(x, f(x))$ and the $x$-axis. At the
point $x_0$ we have $\theta(x_0)=0$ and the total variation of
$\theta$ is just the total curvature which is at most $1$. Hence
$|f'(\theta)|\leq \tan 1 < 2$.

The curvature at the point $(x,f(x))$ is
$$ \kappa(x) = \frac{|f''(x)|}{(1+f'(x)^2)^{3/2}}
$$
Since $|f'|<1$, the second derivative $f''$ and the curvature
$\kappa(x)$ are commensurable and so $|f''(x)|$ is commensurate with
$\kappa$:
\begin{equation}
  |f''(x)| \approx \kappa
\end{equation}

We claim that the width of $C$ is the value of $f$ at the critical
point $x_0$:
 \begin{equation}
   w(C) = f(x_0)
 \end{equation}
To see this, note that the supporting line $L_1(t)$ of $C$ at the point
$(t,f(t))$ for $0<t<L$ is the tangent line
 $$L_1(t): y=f'(t)x + f(t)-tf'(t)$$
At $t=x_0$ this is the line $y=f(x_0)$ and the other supporting line
$L_2(x_0)$ of $C$ parallel to it is the $x$-axis $y=0$, and $f(x_0)$ is the
distance between these two lines. For $0<t<x_0$ the other supporting line
$L_2(t)$ parallel to $L_1(t)$ goes through the end point $(L,0)$ of the arc,
with equation
$$ L_2(t) : y=f'(t)(x-L) ,\qquad 0<t<x_0 $$
while for $x_0<t<L$, the line $L_2(t)$ passes through the origin
$(0,0)$ with equation
$$ L_2(t): y=f'(t)x,\qquad x_0<t<L $$
Hence the distance between $L_1(t)$ and $L_2(t)$ is
\begin{equation}
  d(t) = \frac{g(t)}{\sqrt{1+f'(t)^2}}, \quad g(t) =
  \begin{cases}
    f(t) + (L-t)f'(t),& 0<t<x_0 \\ f(t)-tf'(t),& x_0<t<L
  \end{cases}
\end{equation}

Since $|f'(t)|< 2$, this shows that
\begin{equation}
  d(t) \approx g(t)
\end{equation}
and it suffices to show that
$$g(t) \geq g(x_0)=f(x_0)$$
If $0<t<x_0$ then $g'(t)=(L-t)f''(t)<0$ so $g$ is decreasing and so
$g(t) >g(x_0) = f(x_0)$, while if $x_0<t<L$ then $g'(t)= -tf''(t)>0$
so $g$ is increasing and so $g(t)>g(x_0) = f(x_0)$.

Having established that $w(C) \approx f(x_0)$, it remains to show that
$$
f(x_0) \approx \kappa \ell^2
$$
Assuming say that $x_0 \leq L/2$, we expand $f$ in a Taylor series
around the endpoint $x=L$ further from $x_0$, finding
\begin{equation}
   0=f(L)= f(x_0)  + f'(x_0)(L-x_0) + \frac 12 f''(y)(L-x_0)^2
\end{equation}
for some $x_0<y<L$. Using $f(L)=0$, $f'(x_0) =0$ and $f''<0$,
$|f''|\approx \kappa$ and $L/2\leq L-x_0\leq L$ we get
$$
f(x_0) = -\frac 12 f''(y)(L-x_0)^2 \approx \kappa L^2
$$
Now  note that $L\approx \ell$ because
$$ \ell = \int_0^L \sqrt{1+f'(t)^2}dt \in [L,3 L]$$
using  $|f'(x)|<2$. Hence $f(x_0)\approx \kappa \ell^2$ as claimed.
\end{proof}

\section{Local estimates on the width}

\subsection{Fourier transforms of arcs}
We  establish some bounds on the Fourier transform of
measures supported by ``regular'' arcs.

Let $\gamma:[0, \ell]\to C$ be an arc-length parameterization of the
regular arc $C$, so that $|\dot\gamma|=1$, and
$\kappa<|\overset{..}\gamma|< 2\kappa$.
Note that if $\xi\in\mathbb R^2, |\xi|= 1$ and $0<\rho<\frac{\ell
  \kappa}{10}$, then
\be\label{def of I_xi}
I_\xi =\{t\in I : |\xi\cdot \dot\gamma(t)|<\rho\}
\ee
is an interval of size at most $O(\rho/\kappa)$.

Indeed, the length of $I_\xi$ can be computed as
\begin{equation}
 \int_{I_\xi} dt = \int_{|u|<\rho} \frac{du}{| \xi \cdot \ddot{\gamma}(t)|}
\end{equation}
one using the change of variable $u= \xi \cdot \dot{\gamma}(t)$.
Denoting by $\theta(t)$ the angle between $\xi$
and the unit tangent $\dot{\gamma}(t)$ to the curve, so that by assumption
$|\cos \theta(t)| = | \xi \cdot \dot{\gamma}(t)  | <\rho$,
we have on noting that $\ddot{\gamma}$ is the normal vector to the curve, that
\begin{equation}
 | \xi \cdot \ddot{\gamma}(t)| = | \ddot{\gamma}(t)| |\sin \theta(t)|
> \kappa \sqrt{1-\rho^2}
\end{equation}
and hence
\begin{equation}
 \length I_\xi < \frac{2\rho}{\kappa \sqrt{1-\rho^2}} \ll \frac \rho \kappa
\end{equation}
since  $\rho<\ell\kappa/10<1/10$.

\medskip

\begin{lemma}\label{Lemma3.1}
Let $\xi\in\mathbb R^2\backslash\{0\}$ and assume
\be\label{3.2}
\Big|\frac \xi{|\xi|} \cdot \dot{\gamma} (t)\Big|>\rho
\text { for all } t\in[0, \ell] \;.
\ee
Let $\omega:\mathbb R\to\mathbb R_+$, $\supp \omega\subset [0,
  \ell]$ satisfy
\be\label{3.3}
\int\omega =1 \text { and } \int|\omega'|<\frac c\ell \;.
\ee
Then
\be\label{3.4}
\Big|\int e\big(\xi\cdot \gamma(t)\big)\omega(t) dt\Big|
< \frac c{\rho|\xi|} \Big(\frac 1\ell+\frac \kappa\rho\Big)
\ee
(where $c$ denotes various constants).
\end{lemma}
%\noindent

\begin{proof}
A change of variables $u=\xi\cdot \gamma(t)$ gives
$$
\begin{aligned}
&\Big|\int e\big(\xi\cdot \gamma(t)\big)\omega(t) dt
\Big|\leq \int\Big|\frac d{dt} \Big[\frac{\omega(t)}{\xi\cdot
    \dot\gamma(t)}\Big]\Big|
dt \leq\int\Big\{\frac{|\dot\omega(t)|}{|\xi\cdot \dot\gamma(t)|}+
\frac{\omega(t)|\xi\cdot \overset{..}\gamma(t)|}{|\xi\cdot \dot
\gamma(t)|^2}\Big\}\\
&\leq \frac 1{\rho|\xi|} \int|\dot\omega(t)| +\frac
\kappa{\rho^2|\xi|} \int\omega(t)< \frac c{\rho\ell|\xi|}+
\frac\kappa{\rho^2|\xi|}
\end{aligned}
$$
 from the assumptions.
\end{proof}

Fix $ E\in\mathbb Z_+$ (large), $\lambda=\sqrt E$ and let $\mathcal
E=\mathbb Z^2 \cap \{|x|=\lambda\}$.
%\medskip
%Let $\gamma:I=[0, \ell]\to C$ satisfy $|\dot\gamma|=1, \kappa<
%|\overset{..}\gamma|< 2\kappa$ where $0<\ell<1$, $\kappa\ell< 1$.
Fix $0<c_0<\frac 1{100}$ and take $\rho=c_0\kappa \ell/|\mathcal E|^{2}$.
We let $\xi$ run over all vectors
$\xi=\frac{\xi_1-\xi_2}{|\xi_1-\xi_2|}$, $\xi_1\not= \xi_2$
in $\mathcal E$.
Excluding the corresponding subintervals $I_\xi$ of
\eqref{def of I_xi} from $I$, of length $|I_\xi|<c_0\ell|\vE|^{-2}$,
we obtain

\begin{lemma}\label{Lemma3.5}
%Let $\gamma:I=[0, \ell]\to C$ be as above.

There is a collection of at most $|\mathcal E|^2$ disjoint
sub-intervals $I_\tau\subset I$ with the following properties:
\be
\label{3.6}
|I_\tau|> c_0|\mathcal E|^{-2}\ell
\ee
\be
\label{3.7}
\sum|I_\tau|>(1-2c_0)\ell
\ee
\be\label{3.8}
\Big|\dot\gamma(t)\cdot \frac{\xi_1-\xi_2}{|\xi_1-\xi_2|}\Big|> c_0
 \frac{\kappa\ell}{|\vE|^2}\quad  \text { for } \xi_1\not= \xi_2\text
{ in $\mathcal E$ and $t\in I$}
\ee
%\be \label {3.9}\text{
Let $\omega: \mathbb R\to\mathbb R_+$,  $\supp \omega\subset I_\tau$ satisfy
%\ee
\be
\label{3.10}
\int\omega=1 \text { and } \int|\dot\omega|\lesssim
\frac{|\vE|^2}{c_0\ell} \;.
\ee

Then for all $\xi_1\not= \xi_2$ in $\mathcal E$
\be\label{3.11}
\Big|\int e\big((\xi_1-\xi_2)\cdot \gamma(t)\big)\omega(t) dt\Big|
\lesssim \frac{|\mathcal E|^4}{c_0^2 \omega (C)} \ \frac 1{|\xi_1-\xi_2|}
\ee
where $\omega(C)=\ell^2\kappa$ is the width of $C$.
\end{lemma}

The estimate \eqref{3.11} follows indeed from \eqref{3.4} and the above
choice of $\rho$.
\medskip

Returning to Theorem~\ref{Theorem1}, we simply replace $I$ by some
$I_\tau$ and $C$ by $C_\tau=\gamma(I_\tau)$.
Redefining $\ell=|I_\tau|$, we have for all $\xi_1\not=
\xi_2\in\mathcal E$ the estimate
\be\label{3.12}
\Big|\int e\big((\xi_1-\xi_2)\cdot \gamma(t)\big)\omega(t) dt\Big|\ll
\frac{\lambda^\ve}{\omega(C) |\xi_1-\xi_2|}
\ee
if $\omega:\mathbb R\to\mathbb R_+$,  $\supp \omega\subset I$ satisfies
\be\label{3.13}
\int\omega=1, \quad \int|\dot\omega|\lesssim \frac 1\ell \;.
\ee

\subsection{The exponent $1/6$} \label{subsec:warmup}
As a warm-up, we show how to prove Conjecture~\ref{Conjecture1} for
almost all energies $E$ and how to obtain a weaker version of
Theorem~\ref{Theorem1} with the exponent $1/6$ instead of $1/3$.

Consider the Fourier expansion of $\vp$:
\be\label{3.14'}
\vp(x)=\sum_{\xi\in\mathcal E}\^\vp(\xi) e(x\cdot \xi).
\ee
Since the Fourier coefficients of $\vp$ satisfy
$\sum_{\xi\in \vE} |\^\vp(\xi)|^2 = ||\vp||_2^2$, we have
$|\^\vp(\xi)|\leq ||\vp||_2$ for
all $\xi\in \vE$ and hence there is some $\xi_0$ for which
$$|\^\vp(\xi_0)|\geq \frac{||\vp||_2}{\sqrt{|\vE|}}$$
Replacing $\vp$ by $\vp/\^\vp(\xi_0)$, we may thus assume
\begin{equation}
  \^\vp(\xi_0) = 1, \quad ||\vp||_2 \leq   \sqrt{|\vE|}
\end{equation}
and in particular $|\^\vp(\xi)|\ll \lambda^\epsilon$ for all $\epsilon>0$.

Assume $C\subset \nodal_\vp$.
Since $\vp\big(\gamma(t)\big)=0$, we obtain for any weight function
$\omega$ as in Lemma~\ref{Lemma3.1} that
\be
\begin{split}
0&=\int_I \vp\big(\gamma(t)\big) e\big(-\xi_0\cdot \gamma(t)\big) \omega(t)dt\\
&=1 + \sum_{\xi\neq \xi_0} \^\vp(\xi) \int_I e\big((\xi-\xi_0)\cdot
  \gamma(t)\big) \omega(t) dt
\end{split}
\ee
By Jarnik, there is at most one frequency $\xi_1\neq \xi_0$ at
distance $\ll \lambda^{1/3}$ from $\xi_0$. For all other frequencies we
use \eqref{3.11}
%to bound the integral $\int e\big((\xi-\xi_0)\cdot
%  \gamma(t)\big) \omega(t) dt $
together with $\sum_\xi |\^\vp(\xi)|^2
  \leq |\vE|\ll \lambda^\epsilon$ to get
\be\label{warmup ineq}
1+ \^\vp(\xi_1) \int e\big((\xi_1-\xi_0)\cdot \gamma(t)\big)
\omega(t)dt
   \ll
\frac{\lambda^\epsilon}{w}\sum_{\xi\neq \xi_0,\xi_1}
\frac {|\^\vp(\xi)|}{|\xi-\xi_0|}
\ll \frac{\lambda^{-1/3+\epsilon}}{w}
\ee

We may now show that Conjecture~\ref{Conjecture1} holds for almost all
$E$. First choose $E=\lambda^2$ satisfying \eqref{2.10}. Then $\xi_1$ does
not exist and $|\xi-\xi_0|> \lambda^{1-\epsilon}$ for $\xi\neq \xi_0$,
hence \eqref{warmup ineq} gives $1\ll \lambda^{-1+\epsilon}/w$, that
is $w\ll \lambda^{-1+\epsilon}$.

Returning to the case of general $E$,
if there is no such $\xi_1$, that is if $\xi_0$ is at distance at
least $\lambda^{1/3}$ from all other frequencies, then \eqref{warmup ineq}
implies $w\ll \lambda^{-1/3+\epsilon}$.

Otherwise, that is if there is a neighbour $\xi_1$, we proceed as
follows:
Start by performing a rotation $T$ of the plane as to insure
\be\label{3.20'}
T(\xi_1-\xi_0)=(|\xi_1-\xi_0|, 0)\in\mathbb R^2.
\ee
Denoting $T\gamma$ again by $\gamma=(\gamma_1,\gamma_2)$,
we obtain from \eqref{warmup ineq} that
\be
1+ \^\vp(\xi_1) \int e\big((|\xi_1-\xi_0| \gamma_1(t)\big)
\omega(t)dt
   \ll \frac{\lambda^{-1/3+\epsilon}}{w}
\ee

Next we specify $\omega$.
%\medskip
Writing $\dot\gamma(s) =e^{i\theta(s)}$, we have $\dot\theta (s) \sim
\kappa$ (or $-\kappa$, which is similar)
and
$$
\overset{..}\gamma_1(s)= -\big(\sin\theta(s)\big)\dot\theta (s) \sim
-\sqrt{1-\dot\gamma_1(s)^2}\kappa.
$$
Therefore there is a suitable restriction of $s\in I_1\subset I$,
$|I_1|\sim\ell$ and some $\rho\gtrsim \kappa\ell$
(recall that $\kappa\ell<1$) such that
\be\label{3.25'}
|\dot\gamma_1(s)-\rho|<c\kappa\ell<\frac \rho{10} \text { for } s\in I_1.
\ee

Let $s_0\in I_1$ be the center of $I_1$.
Define
\be\label{3.26'}
\omega (s)= \frac{\dot\gamma_1(s) \eta\big(\gamma_1(s)
  -\gamma_1(s_0)\big)}{\int_{I_1} \dot\gamma_1(s) \eta\big(\gamma_1(s)
-\gamma_1(s_0)\big) ds}
\ee
where $\eta$ is a bump-function of the form $\eta(x) = \eta_0(\frac
x{\rho \ell})$ with $\eta_0\geq 0$, $\int \eta_0 = 1$ (see
Figure~\ref{bumpfunctonfig}),
\begin{figure}[h]
\begin{center}
  \includegraphics[width=70mm]%[width=130mm]
{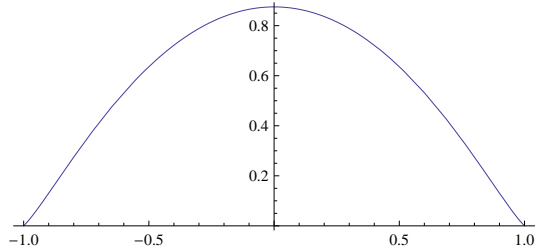}
 \caption{The bump function $\eta_0(t)$}
\label{bumpfunctonfig}
\end{center}
\end{figure}
chosen as to ensure that $\supp\omega \subset I_1$ (we use
\eqref{3.25'} here). Also
$$
0\leq\omega<\frac c\ell, \quad \int\omega=1
$$
and
$$
\int|\omega'|\lesssim \frac\kappa\rho+\frac 1\ell\lesssim \frac 1\ell
$$
and \eqref{3.13} holds.
%\begin{figure}

%\input fig2.tex

%\caption{A bump function}
%\end{figure}

%\smallskip

With the choice \eqref{3.26'} and change of variable $u=\gamma_1(s)
-\gamma_1(s_0)$, one obtains in \eqref{3.23}
\be\label{3.27'}
\begin{split}
\int e\big(|\xi_1-\xi_0| \gamma_1 (t)\big)\omega(t) dt &=
e\big(|\xi_1-\xi_0|\gamma_1 (s_0)\big)
\frac{\int e(|\xi_1-\xi_0| u)\eta(u)du}{\int\eta (u) du}\\%[6pt]
&= e\big(|\xi_1-\xi_0|\gamma_1 (s_0)\big)\frac 1{\rho\ell} \int e(|\xi_1-\xi_0|
u)\eta(u) du\\%[6pt]
%\end{aligned}
%$$
%\be\label{3.27}
& = e\big(|\xi_1-\xi_0| \gamma_1(s_0)\big) \int e(\rho\ell
|\xi_1-\xi_0|t)\eta_0(t)dt
\end{split}
\ee
since $\eta_0(t) =\eta(\rho\ell t)$. Thus we find
\be
1+  e\big(|\xi_1-\xi_0|\gamma_1(s_0)\big) b
 \ll\frac{\lambda^{-1/3+\epsilon}}{w}
\ee
where
\be
b= \^\eta_0(|\xi_1-\xi_0|\rho\ell) \^\vp(\xi_1)
\ee
satisfies $|b|\ll \lambda^\epsilon$.

Note that our choice of $s_0\in I_1$ allows moving $s_0$ within an
interval $I_2\subset I_1$ of
size $|I_2|=\frac 12 |I_1|\sim\ell$.
Since $\gamma_1(I_2)$ contains an interval of size at least
$\sim\rho\ell\gtrsim w$, it follows that
\be\label{3.34'}
\max_{s_0\in I_2} \Big|1 +b e\big(|\xi_1-\xi_0| \gamma_1 (s_0)\big)  \Big|
\geq
\max_{u\in U} \Big|1+b  e(u) \Big|
\ee
where $U\subset\mathbb R$ is some interval of size
$\sim |\xi_1-\xi_0|w$.
Then we have
\be%\label{3.35'}
\eqref{3.34'} \geq \frac 12 \min (1, w |\xi_1-\xi_0|)
\ee
Indeed, if $|b|\geq 3/2$ or $|b|\leq 1/2$ then $|1+b  e(u)|\geq 1/2$,
while if $1/2<|b|<3/2$ then  we can bound
$$
\max_{u\in U} |1+b  e(u) |\geq |b|\max_{u\in U} |\sin u|\geq \frac 14|U|
$$
Thus we find
\be
\min(1, w|\xi_1-\xi_0|) \ll \frac{\lambda^{-1/3+\epsilon}}{w}
\ee
If the minimum is $1$, we get $w\ll\lambda^{-1/3+\epsilon}$. Otherwise
(taking into account $|\xi_1-\xi_0|\geq 1$) we get
$$
w\ll \lambda^{-1/6+\epsilon}
$$

\subsection{Proof of Theorem~\ref{Theorem1}}

%\be\label{3.14}
%\vp(x)=\sum_{\xi\in\mathcal E}\^\vp(\xi) e(x\cdot \xi).
%\ee
Fix some $\xi_0\in \mathcal E$ and enumerate $\mathcal E=\xi_0, \xi_1, \ldots$ such that
\be\label{3.15}
|\xi_0-\xi_j|\leq |\xi_0-\xi_{j+1}|.
\ee
Write
\be\label{3.14}
\vp(x) =\sum_j c_j e(x\cdot \xi_j).
\ee
Let $1< r<\frac 1{10} \lambda$ be a parameter and take $J\in\mathbb Z_+$ with
\be\label{3.16}
|\xi_0-\xi_J|\leq r, \quad |\xi_0 -\xi_{J+1}|\geq r \;.
\ee

Assume $C\subset \nodal_\vp$.
Then since $\vp\big(\gamma(t)\big)=0$, we obtain for any weight function
$\omega$ as in Lemma~\ref{Lemma3.1} that
\be
\begin{aligned}\label{3.17}
0&=\int\vp\big(\gamma(t)\big) e\big(-\xi_0\cdot \gamma(t)\big) \omega(t)dt\\
&=c_0+\sum_{1\leq j\leq J} c_j \int e\big((\xi_j-\xi_0)\cdot
  \gamma(t)\big) \omega(t) dt \\
\end{aligned}
\ee
\be\label{3.18}
+0\Big( \frac{\lambda^\ve} w \sum_{j>J}\frac {|c_j|}{|\xi _0-\xi_j|}\Big)
\ee
%denoting $w=w(C)$.

%\medskip
%The purpose of what follows is to make a further estimate of
%\eqref{3.17} by a careful choice of $\omega$.

Perform a rotation $T$ of the plane as to insure
\be\label{3.19}
T(\xi_J-\xi_0) = (\xi'_J,0), \quad \xi'_J= |\xi_J-\xi_0|
\ee
and denote
\be\label{3.20}
T(\xi_j-\xi_0)=(\xi_j', \zeta_j')\in\mathbb R^2.
\ee
Clearly
\be\label{3.21}
\xi_j'\sim |\xi_j-\xi_0| ,\quad j\leq J
\ee
and
\be\label{3.22}
|\zeta_j'|< 2\frac r\lambda |\xi_j-\xi_0|.
\ee

Denoting $T\gamma$ again by $\gamma=(\gamma_1,\gamma_2)$, we easily obtain
\begin{eqnarray}\label{3.23}
\eqref{3.17} &=&
c_0 +\sum_{1\leq j\leq J}\Big[\int e\big(\xi_j' \gamma_1 (t)\big)
\omega(t) dt\Big]c_j\\
\label{3.24}
&+& 0\Big(\frac r\lambda \sum_{0<j<J} |c_j| \ |\xi_j-\xi_0|\Big).
\end{eqnarray}

Next we specify $\omega$ as in \S~\ref{subsec:warmup}, by picking a
subinterval $I_1\subset I$ $|I_1|\approx \ell$, such that
\be\label{3.25}
|\dot\gamma_1(s)-\rho|<c\kappa\ell<\frac \rho{10} \text { for } s\in I_1.
\ee
for some $\rho\gtrsim \kappa\ell$ (recall that $\kappa\ell<1$).
%Writing $\dot\gamma(s) =e^{i\theta(s)}$, we have $\dot\theta (s) \sim
%\kappa$ (or $-\kappa$, which is similar)
%and
%$$
%\overset{..}\gamma_1(s)= -\big(\sin\theta(s)\big)\dot\theta (s) \sim
%-\sqrt{1-\dot\gamma_1(s)^2}\kappa.
%$$
%Therefore there is a suitable restriction of $s\in I_1\subset I$,
%$|I_1|\sim\ell$ and some $\rho\gtrsim \kappa\ell$
%(recall that $\kappa\ell<1$) such that
%\be\label{3.25}
%|\dot\gamma_1(s)-\rho|<c\kappa\ell<\frac \rho{10} \text { for } s\in I_1.
%\ee
Let $s_0\in I_1$ be the center of $I_1$.
Define
\be\label{3.26}
\omega (s)= \frac{\dot\gamma_1(s) \eta\big(\gamma_1(s)
  -\gamma_1(s_0)\big)}{\int_{I_1} \dot\gamma_1(s) \eta\big(\gamma_1(s)
-\gamma_1(s_0)\big) ds}
\ee
where $\eta$ is a bump-function of the form $\eta(x) = \eta_0(\frac
x{\rho \ell})$ with $\eta_0\geq 0$, $\int \eta_0 = 1$,
%\input fig2.tex
%\smallskip
%\noindent
chosen as to ensure that $\supp\omega \subset I_1$ (we use
\eqref{3.25} here). Also
$$
0\leq\omega<\frac c\ell, \quad \int\omega=1
$$
and
$$
\int|\omega'|\lesssim \frac\kappa\rho+\frac 1\ell\lesssim \frac 1\ell
$$
and \eqref{3.13} holds.

With the choice \eqref{3.26} and change of variable $u=\gamma_1(s)
-\gamma_1(s_0)$, one obtains in \eqref{3.23}
\be\label{3.27}
\begin{split}
\int e\big(\xi_j' \gamma_1 (t)\big)\omega(t) dt &=
e\big(\xi_j'\gamma_1 (s_0)\big)
\frac{\int e(\xi_j' u)\eta(u)du}{\int\eta (u) du}\\%[6pt]
&= e\big(\xi_j'\gamma_1 (s_0)\big)\frac 1{\rho\ell} \int e(\xi_j'
u)\eta(u) du\\%[6pt]
%\end{aligned}
%$$
%\be\label{3.27}
& = e\big(\xi_j' \gamma_1(s_0)\big) \int e(\rho\ell \xi_j't)\eta_0(t)dt
\end{split}
\ee
since $\eta_0(t) =\eta(\rho\ell t)$.

\medskip

Therefore
$$
|\eqref{3.27}|< \lambda^{-100} \text { unless } |\xi_j'|< \frac{\lambda^\ve}{\rho\ell}.
$$
Hence, in \eqref{3.23}
\be\label{3.28}
\Big|\int e\big(\xi_j' \gamma_1(t)\big)
\omega(t) dt\Big| <\lambda^{-100} \text{ unless } |\xi_j'|<\frac{\lambda^\ve}{w(C)}
\ee
and from \eqref{3.17}, \eqref{3.18}
\be\label {3.29}
\Big|c_0+\sum_{0<|\xi_j-\xi_0|<\frac{\lambda^\ve}{w}} c_j'\Big[\int e(\rho\ell\xi_j' t)\vp_0 (t)dt\Big]\Big|<
\eqref{3.18}+\eqref {3.24} +\lambda^{-\frac 1{100}}
\ee
where
\be\label{3.30}
c_j' =e\big(\xi_j' \gamma_1(s_0)\big)c_j.
\ee

Arguing by contradiction, assume
\be\label{3.31}
w(C)>\lambda^{-\frac 13+ 5\ve} \;.
\ee
Since $|\xi_2-\xi_0|\geq \lambda^{\frac 13}$, the restriction
$0<|\xi_j-\xi_0|<\frac {\lambda^\ve}{w}$ excludes all
terms, except possibly $j=1$.
Hence, either
\be\label{3.32}
|c_0|< \eqref{3.18}+\eqref {3.24} + \lambda^{-100}
\ee
or
\be\label{3.33}
\Big|c_0+c_1'\Big[\int e(\rho\ell\xi_1' t)\vp_0 (t) dt\Big]\Big|< \eqref{3.18} +\eqref{3.24} +\lambda^{-100}.
\ee

If \eqref{3.33}, we argue as follows.
Recall \eqref{3.30} and note that our choice of $s_0\in I_1$ allows moving $s_0$ within an interval $I_2\subset I_1$ of
size $|I_2|=\frac 12 |I_1|\sim\ell$.
Since $\gamma_1(I_2)$ contains an interval of size at least $\sim\rho\ell\gtrsim w$, it follows that
\be\label{3.34}
\begin{aligned}
&\max_{s_0\in I_2} \Big|c_0+c_1 e\big(\xi_1' \gamma_1 (s_0)\big) \Big[\int e(\rho\ell\xi_1't)\vp_0(t)dt\Big]\Big|\geq\\[6pt]
&\max_{u\in U} \Big|c_0+c_1 e(u) \Big[\int e(\rho\ell\xi_1' t)\vp_0(t) dt\Big]\Big|
\end{aligned}
\ee
where $U\subset\mathbb R$ is some interval of size
$\sim|\xi_1'|w=|\xi_1-\xi_0|w$.
Clearly
\be\label{3.35}
\eqref{3.34} > c\min (1, w |\xi_1-\xi_0|) |c_0|.
\ee
\medskip

In summary, we proved that
\be\label{3.36}
\min (1, w |\xi_1-\xi_0|)|c_0|< \eqref{3.18}+\eqref {3.24}+\lambda^{-100}.
\ee
Taking $r=|\xi_0-\xi_J|$ in \eqref {3.36} gives
\begin{lemma}\label{Lemma3.37}
Fix $\xi_0\in\mathcal E$ and enumerate $\mathcal E=\{\xi_j\}$
according to \eqref{3.15}.
Assuming $w(C)>\lambda^{-\frac 13+ 5\ve}$,
for $J\geq 1$, one has the bound
\be\label{3.38}
\begin{aligned}
&\min(1, w|\xi_1-\xi_0|) |c_0|<  \\[8pt]
& 0 \Big\{\frac{|\xi_J-\xi_0|}{\lambda} \Big(\sum_{0<j<J}
|c_j| \ |\xi_j -\xi_0|\Big)
+\frac {\lambda^\ve}{w} \Big(\sum_{j>J} \frac{|c_j|}{|\xi_j-\xi_0|}\Big)\Big\}
+\lambda^{-100}
\end{aligned}
\ee
%under assumption \eqref{3.31}.
%where  $c_j=\^\vp(\xi_j)$.
\end{lemma}

This is our main estimate.
\medskip

We apply \eqref {3.38} with $J=1$ and $J=2$, obtaining the inequalities
\be\label{3.39}
\min (1, w|\xi_1-\xi_0| ) |c_0| < \frac{\lambda^\ve}w \ \frac 1{|\xi_2-\xi_0|}+\lambda^{-100}
\ee
and
\be\label{3.40}
\min(1, w|\xi_1-\xi_0|) |c_0| < c\frac{|\xi_1-\xi_0| \ |\xi_2 -\xi_0|}{\lambda} +\frac{\lambda^\ve}w \sum_{j\geq 3}
\frac{|c_j|}{|\xi_j-\xi_0|}+\lambda^{-100}.
\ee

Start by taking $\xi_0\in\mathcal E$ such that $|c_0|\geq \frac
1{|\mathcal E|^{1/2}}$ (we normalize $\Vert\phi\Vert_2=1$).
 From \eqref {3.39}
$$
w< \frac{\lambda^\ve}{|\xi_2-\xi_0|} +\frac {\lambda^\ve}{(|\xi_1-\xi_0| \ |\xi_2-\xi_0|)^{\frac 12}}
$$
and by \eqref {3.31}, since $|\xi_2- \xi_0|\gtrsim \lambda^{1/3}$, it follows
\be\label{3.41}
|\xi_1-\xi_0| \ |\xi_2 -\xi_0|<\frac {\lambda^{2\ve}}{w^2} < \lambda^{\frac 23 -\ve}.
\ee
 From \eqref {3.40}, \eqref{3.41}, either
$$
w<\frac{\lambda^\ve}{|\xi_3-\xi_0|} +\lambda^{-100} <\lambda^{-\frac 13+\ve} \ \text {(contradicting \eqref{3.31})}
$$
or
$$
w<\lambda^\ve\frac{|\xi_2-\xi_0|}\lambda + \frac{\lambda^\ve}{w|\xi_1-\xi_0|}\Big(\sum_{j\geq 3} \ \frac
{|c_j|}{|\xi_j-\xi_0|}\Big)
$$
and therefore
\be\label{3.42}
w^2 <\frac{\lambda^\ve|c_j|}{|\xi_1 -\xi_0| \ |\xi_j-\xi_0|}  \ \text { for some $j\geq 3$.}
\ee

Next, we apply \eqref{3.39} replacing $\xi_0$ by $\xi_j$.
Enumerate
$$\mathcal E=\{\xi_{j, k}; k=0, 1, \ldots\}$$
where $\xi_{j, 0}=\xi_j$ and
\be\label{3.43}
|\xi_j-\xi_{j, k}| \leq |\xi_j-\xi_{j, k+1}|.
\ee
We obtain
$$
\min (1, w|\xi_{j, 1} -\xi_j|)|c_j| <\frac{\lambda^\ve}w\ \frac 1{|\xi_{j, 2}-\xi_j|} +\lambda^{-100}
$$
and
\be\label{3.44}
|c_j|<\frac{\lambda^\ve}{w|\xi_{j, 2}-\xi_j|}+\frac{\lambda^\ve}{w^2|\xi_{j, 1}-\xi_j| \ |\xi_{j, 2}-\xi_j|}
+\lambda^{-10}.
\ee
\medskip

Substituting \eqref{3.44} in \eqref {3.42}, it follows that either
\be\label{3.45}
w^3<\frac{\lambda^\ve}{|\xi_1 -\xi_0| \ |\xi_j -\xi_0| \ |\xi_{j, 2}-\xi_j|}
\ee
or
\be\label{3.46}
w^4<\frac{\lambda^\ve}{|\xi_1-\xi_0| \ |\xi_j-\xi_0| \ |\xi_{j, 1} -\xi_j| \ \xi_{j, 2} -\xi_j|}.
\ee
Note that obviously $|\xi_{j, 1} -\xi_j|, |\xi_{j, 2}-\xi_j|\leq \max
(|\xi_j-\xi_0|, |\xi_j-\xi_1|)$.

\medskip

We distinguish two cases

\medskip

\noindent
{\bf Case 1.} $\xi_{j, 1} \not\in \{\xi_0, \xi_1\}$.
\medskip

The points $\xi_0, \xi_1, \xi_j, \xi_{j, 1}$ are distinct elements of $\mathcal E$ on an arc of size \hfill\break
$r<8|\xi_j -\xi_0|$.
Lemma~\ref{Lemma2.3} implies that
\be\label{3.47}
|\xi_1-\xi_0| \ |\xi_{j, 1} -\xi_j| \ |\xi_j -\xi_0|\gtrsim \lambda
\ee
and therefore $\eqref{3.45} < \lambda^{\ve-1}, \eqref{3.46} < \lambda^{\ve-4/3}$, a contradiction.

\medskip

\noindent
{\bf Case 2.} $\xi_{j, 1} \in \{\xi_0, \xi_1\}$.
\medskip

Since $|\xi_1-\xi_0|<\frac 1{10} \lambda^{\frac 13}$ by \eqref {3.41}, it follows that
$$
|\xi_j-\xi_{j, 2}|\geq |\xi_j-\xi_{j, 1}|\geq |\xi_j-\xi_0|-|\xi_1-\xi_0|>\frac 12 |\xi_j-\xi_0|.
$$
Hence
\be\label{ineq1}
\eqref{3.45} < \frac{\lambda^\ve}{|\xi_1-\xi_0| \ |\xi_j-\xi_0|^2}
\leq \frac{\lambda^\ve}{|\xi_1-\xi_0| \ |\xi_2-\xi_0|^2}<
\lambda^{\ve-1}
\ee
 by Lemma \ref{Lemma2.1}, and
\be\label{ineq2}
\eqref{3.46} < \frac{\lambda^\ve}{|\xi_1-\xi_0| \ |\xi_j-\xi_0|^3}<
\lambda^{\ve-\frac 43}
\ee
which is again a contradiction.
This completes the proof of Theorem~\ref{Theorem1}. \qed

\medskip
Note that \eqref{ineq1}, \eqref{ineq2} could be saturated, since $\xi_0,
\xi_1, \xi_2, \xi_3$ could lie on an arc of size $\approx
\lambda^{1/3}$; cf \cite{C-G2} for a discussion of this phenomenon.

%\medskip
%We conclude this section by verifying the validity of
%Conjecture~\ref{Conjecture1} for most energies $E$.
%\medskip
%
%Choose $E$ satisfying \eqref{2.10}.
%Applying \eqref {3.17}, \eqref {3.18} with $r=\lambda^{1-\ve}$, the second term in \eqref {3.17} is not present and we get
%\be\label{3.48}
%|c_0|< \frac {\lambda^\ve} w\sum_{|\xi_0-\xi_j|> r} \ \frac
%{|c_j|}{|\xi_0-\xi_j|} < \frac {\lambda^{2\ve}}{wr}\;.
%\ee
%Letting $|c_0|\geq |\mathcal E|^{-\frac 12}$, we find that indeed
%$$
%w(C) <\lambda^{-1+3\ve}
%$$
%as claimed.

\section
{Local length estimates and the Donnelly-Fefferman doubling exponent}

\subsection{The doubling exponent for the torus}
We will apply the results from \cite{D-F} in the particular setting
$M=\mathbb T^2$.
An additional ingredient is an estimate on the doubling exponent
\be\label{4.1}
\beta(\vp)=\max_B\log\Big(\frac {\max_B|\vp|}{\max_{\frac 12 B}|\vp|}\Big)
\ee
where $B\subset M$ is an arbitrary disc and $\frac 12B$ denotes the disc with same center and half radius.

As shown in \cite{D-F}, assuming $M$ is $C^\infty$-smooth and
$-\Delta\vp= E\vp$, $E=\lambda^2$, one has a general bound
\be\label{4.2}
\beta(\vp)< C_M\lambda.
\ee
It turns out that for $M=\mathbb T^2$, \eqref{4.2} can be considerably improved.

\begin{lemma}\label{Lemma4.3}
For $M=\mathbb T^2, -\Delta\vp =\lambda^2\vp$ and $\mathcal E=\mathbb Z^2\cap
\{|x|=\lambda\}$, one has
\be
\label{4.4}
\beta(\vp)< C|\mathcal E| <\exp c\frac{\log \lambda }{\log\log \lambda}
\ee
(taking $\lambda>10$ say).
\end{lemma}

Based on \eqref {4.4}, some of the statements in \cite{D-F} may then be strengthened in the situation $M=\mathbb T^2$.
\medskip

%\subsection{An uncertainty principle}
Lemma \ref{Lemma4.3} is a consequence of a general principle,
an extension of Turan's lemma, for which  
%also known as the `uncertainty principle'.  
we refer to Nazarov \cite{N}: % for the following formulation.

\begin{lemma}\label{Lemma4.5}
Let $f(t)=\sum^J_{j=1} a_j e(\xi_j t), t\in\mathbb R$, where $\xi_1< \xi_2<\cdots<\xi_J\in\mathbb R$.
Let $I\subset\mathbb R$ be an interval and $\Omega\subset I$ a measurable subset. Then
\be\label{4.6}
\sup_{t\in\Omega} |f(t)|> \Big( c\frac{|\Omega|}{|I|}\Big)^{J-1} \sup_{t\in I} |f(t)|.
\ee
\end{lemma}

A simple %easy 
argument based on one-dimensional sections allows one to deduce a
multivariate version of  Lemma~\ref{Lemma4.5} (see e.g. \cite{FM}):
\begin{lemma}
%{\bf Lemma {\ref{Lemma4.5}$'$.}}
\label{Lemma4.5'}
Let $f(x) =\sum^J_{j=1} a_j e(\xi_j\cdot  x)$, $x\in\mathbb R^n$ and $\xi_1,
\cdots, \xi_J\in\mathbb R^n$ be distinct frequencies.
Let $I\subset\mathbb R^n$ be a cube and $\Omega\subset I$ a measurable subset.
Then \eqref{4.6} holds.
\end{lemma}
\medskip

%We leave the argument to the reader.
%\marginpar{check, compare with lit.}
%\medskip

Applying Lemma~\ref{Lemma4.5'} to $f=\vp$ with $J\leq \# \mathcal E$,
it follows that
\be \label{4.7}
\frac{\max_{B}|\vp|}{\max_{\frac 12 B}|\vp|}< C^{[\# \mathcal E]}
\ee
for all discs $B\subset\mathbb T^2$ and \eqref{4.4} follows.
\medskip

The following upper bound on the length of the nodal set lying in sets
of size $\approx \frac 1\lambda$ can be deduced\footnote{
Proposition 6.7 of \cite{D-F} gives an upper bound on the length of the
nodal set in terms of the doubling property for a complex ball, at the
scale of $1/\lambda$. 
To relate this to the doubling exponent $\beta$ of \eqref{4.1}, one
uses a hypo-elliptic estimate \cite[bottom of page 180]{D-F} to
relate the supremum over a complex ball to that over real balls.
Then one can invoke Lemma~\ref{Lemma4.3} to bound $\beta$.} 
 from \cite[Proposition 6.7]{D-F}: 
\begin{lemma}\label{Lemma4.8}
For any disc $B_{\frac 1\lambda}\subset\mathbb T^2$ of size $\frac
1\lambda$,
$$
\hd(\nodal_\vp\cap B_{\frac 1\lambda}) <C [\# \mathcal E] \frac1\lambda\ll \lambda^{\ve-1}
$$
\end{lemma}

%This follows by applying the upper bound of 
%\cite[Proposition 6.7]{D-F} to the eigenfunction on the scale of
%$1/\lambda$. 
%\marginpar{How does prop 6.7 imply this ?}

We will also need the lower bound \cite{D-F}, \S7.

\begin{lemma}\label{Lemma4.9}
There are constants $a>0$, $c>0$ 
so that if we partition $\mathbb T^2$ into squares of size $\frac a\lambda$, 
\be
\label{4.10}
\mathbb T^2 =\bigcup_\nu Q_\nu
\ee
then
\be\label{4.11}
\hd(\nodal_\vp\cap Q_\nu)> c\lambda^{-1}
\ee
holds for at least half of the $Q_\nu$'s.
\end{lemma}

Let us point out that both Lemmas~\ref{Lemma4.8} and \ref{Lemma4.9}
use methods from analytic function theory and hence require $M$ to be
real analytic.
\medskip

We derive one more consequence of Lemma \ref{Lemma4.5'} and
Lemma~\ref{Lemma2.8}.

\begin{lemma}\label{Lemma4.12}
Let $\psi=\sum_{\xi\in\mathcal E'} \^\psi (\xi) e(x \cdot \xi)$ (a
complex trigonometric polynomial) where $\mathcal E'\subset\mathcal
E=\mathcal E_\lambda$ is contained in an arc of size $\lambda^{\frac 12-\sigma} , \sigma>0$.
Let $\Omega\subset\mathbb T^2$ be a measurable set.
Then
\be
\label{4.13}
\sup_{x\in \Omega} |\psi (x)|>(c|\Omega|)^{\frac 1\sigma}\Vert\psi\Vert_\infty.
\ee
\end{lemma}

Note that if Conjecture~\ref{Conjecture2} were true, one could
conclude that in the previous setting
\be\label{4.14}
\sup_{x\in \Omega} |\psi(x)|>(c|\Omega|)^{C(\sigma)} \Vert\psi\Vert_\infty
\ee
if $\mathcal E'$ is contained in an arc of size $\lambda^{1-\sigma}$, $\sigma>0$.
\medskip

\subsection{A Jensen type inequality}
In the spirit of \eqref{4.13}, \eqref{4.14}, one can show that
eigenfunctions of $\mathbb T^2$ can not be too small on large subsets
of $\mathbb T^2$, as a consequence of the following Jensen type
inequality.

\begin{lemma}\label{Lemma4.15}
If $\vp$ is an eigenfunction of $\mathbb T^2$, then
\be\label{4.16}
\int_{\mathbb T^2} \log|\vp(x) |dx \geq \max_{\xi\in\mathbb Z^2} (\log |\^\vp (\xi)|).
\ee
\end{lemma}

This property generalizes to  eigenfunctions on higher
dimensional tori with the same argument.

\begin{proof}
Let $\Delta\vp=-\lambda^2\vp$ and $\vp=\sum_{|\xi|=\lambda}\^\vp
(\xi) e(x \cdot \xi)$.
Fix $\xi_0 \in\mathbb Z^2, |\xi_0|=\lambda$ and consider
$$
\vp(x+\xi_0\theta)= \^\vp(\xi_0) e(x \cdot \xi_0)
e(|\xi_0|^2\theta)+\sum_{\xi\not= \xi_0} \^\vp(\xi) e(x \cdot \xi)
e(\xi \cdot \xi_0\theta)
$$
as a polynomial in $\theta\in\mathbb T$.
Since $\xi\cdot \xi_0\in\mathbb Z$, $\xi \cdot \xi_0<|\xi_0|^2$ for
$\xi\not= \xi_0$, an application of Jensen's inequality to $f(\theta)
=\vp(x+\xi_0\theta)
e(-|\xi_0|^2\theta)$ with fixed $x$, gives
$$
\int\log|\vp (x+\xi_0\theta)|d\theta\geq\int\log
|f(\theta)|d\theta\geq \log |\^\vp(\xi_0)| \;.
$$
Integration in $x\in\mathbb T^2$ implies
$$
\int_{\mathbb T^2}\log|\vp(x)|dx \geq \log |\^\vp(\xi_0)|
$$
proving \eqref{4.16}.
\end{proof}

If we assume $\Vert\vp\Vert_2=1$, then certainly
$\Vert\^\vp\Vert_\infty\geq |\mathcal E|^{-\frac 12}$.
Hence, given any subset $\Omega$ of $\mathbb T^2$, Lemma \ref{Lemma4.15} implies
\be\label{4.17}
\int_\Omega\log|\vp| \geq 
\int_{\mathbb T^2} \log|\vp| -\int_{\mathbb  T^2} \log^+ |\vp| 
\gg - \log|\mathcal  E|-1 \gg -\frac{\log\lambda}{\log\log \lambda}\;. 
\ee

\section{Proof of Theorems \ref{Theorem2} and
  \ref{Theorem3}}\label{sec:5}

Given the eigenfunction $\vp$, $-\Delta\vp=\lambda^2\vp$,
let $\{C_\alpha\}$ be a collection of disjoint regular sub-arcs of
the nodal set $\nodal_\vp$, of width
\be\label{5.1}
w(C_\alpha) >\lambda^{-\rho}
\ee
where $\rho<1$ (we specify $\rho$ later on).
\medskip
Define
\be\label{5.8}
\mathcal N_0:= \bigcup_a C_\alpha
\ee
Our goal is to give an upper bound for the length of $\mathcal N_0$.

For each $C_\alpha$, perform the construction from Lemma \ref{Lemma3.5}, taking $c_0>0$ a small constant, to be specified.
This gives a collection $\{C_{\alpha, \tau}\}$ of sub-arcs of
$C_\alpha$ satisfying in particular
\be\label{5.2}
\sum_\tau|C_{\alpha, \tau}|> (1-2c_0) |C_\alpha|
\ee
\be\label{5.3}
\Big|\int_{C_{\alpha, \tau}} e\big((\xi_1-\xi_2)\cdot x\big)
\frac{ds}{|C_{\alpha, \tau}|}\Big|
\lesssim \frac{|\mathcal E|^4}{c_0^2 w(C_\alpha)} |\xi_1-\xi_2|^{-1} \ll \lambda^{\rho+\ve} |\xi_1-\xi_2|^{-1}
\ee
for all $\xi_1\not= \xi_2$.
Here $ds$ stands for the arc-length measure on $C_\alpha$; $\ve>0$
is arbitrarily small.  We get a subset $\mathcal N_1\subset \mathcal
N_0$  defined by
\be\label{5.9}
\mathcal N_1 :=\bigcup_\alpha\bigcup_\tau C_{\alpha, \tau}.
\ee
Using  \eqref{5.2} and the Donnelly-Fefferman  upper bound \eqref{1.4}
we see that
\be\label{5.10}
\hd(\mathcal N_0\backslash\mathcal N_1)< 2c_0 \sum_\alpha|C_\alpha| \leq 2c_0 \hd(\nodal_\vp)\lesssim c_0\lambda
\ee

Fix $\rho<\rho_1=\rho+3\delta<1$ and introduce a partition
\be\label{5.4}
\mathcal E=\bigcup_\beta\mathcal E_\beta
\ee
of the lattice points $\mathcal E =\mathbb Z^2 \cap\{|\xi|
=\lambda\}$, satisfying
\be\label{5.5}
\dist(\mathcal E_\beta, \mathcal E_{\beta'})>\lambda^{\rho_1} \text { for } \beta\not=\beta'
\ee
\be\label{5.6}
\text{diam\,}\mathcal E_\beta\ll \lambda^{\rho_1+\ve} \text { for each } \beta.
\ee
The construction is straightforward:
If we introduce a graph on $\mathcal E$, defining $\xi \sim\xi'$ if
$|\xi-\xi'|\leq\lambda^{\rho_1}$, its connected components
$\mathcal E_\beta$ are obviously of diameter at most
$\lambda^{\rho_1} \cdot \# \mathcal E \ll \lambda^{\rho_1+\ve}$ and
\eqref{5.5} holds.
\medskip

Let
$$
\vp=\sum\vp_\beta, \quad \Vert\vp\Vert_2 =1
$$
be the corresponding decomposition of our eigenfunction $\vp$.
Thus $\supp \^\vp_\beta\subset\mathcal E_\beta$.
For each $\alpha, \tau$ we have by \eqref{5.3}
\begin{multline}
\label{5.7}
0=\int_{C_{\alpha, \tau}}|\vp(x)|^2=
\sum_\beta\int_{C_{\alpha, \tau}} |\vp_\beta|^2 \\
+O \Big(
\sum_{\beta\not= \beta'}
\sum_{\substack{\xi,\in    \vE_{\beta} \\ \xi'\in \vE_{\beta'} }}
| \^\vp_\beta(\xi)|  | \^\vp_{\beta'}(\xi')|\frac
     {\lambda^{\rho+\ve}|C_{\alpha, \tau}|}{| \xi -\xi'|}
\Big)
\end{multline}
and \eqref{5.5} implies the bound
$\lambda^{\rho-\rho_1+\ve}|C_{\alpha, \tau}|$ on the last term of \eqref {5.7}.
Summing \eqref{5.7} over all $\alpha$, $\tau$ gives
\be\label{5.11}
\sum_\beta\int_{\mathcal N_1} |\vp_\beta|^2 \ll \lambda^{1+\rho-\rho_1+\ve}.
\ee

Since $\sum_\beta\Vert\vp_\beta\Vert^2_2 =1$, one can specify some
$\beta$ such that $||\vp_\beta|| \geq 1/\sqrt{|\vE|}$ and hence
$$\psi:=\frac{\vp_\beta}{\Vert\vp_\beta\Vert_2}$$
has $\Vert\psi\Vert_2=1$ and  satisfies
\be\label{5.12}
\int_{\mathcal N_1} |\psi|^2 \ll \lambda^{1+\rho-\rho_1+\ve}.
\ee
and
$$
\psi=\sum_{\xi\in\mathcal E'} \^\psi (\xi) e(x\cdot \xi)
$$
with
$\mathcal E'=\vE_\beta \subset\mathcal E$
contained in an arc of size $r<\lambda^{\rho_1+\epsilon}$.

Now define
%Define
%$$
%\mathcal N'=\{x\in\mathcal N; |\psi(x)|< \lambda^{-\tau'}\}
%$$
%$(\tau, \tau'>0)$.
%\medskip
\be\label{5.13}
\mathcal N' :=\{x\in\mathcal N, |\psi(x)|<\lambda^{-\delta}\}, \qquad
\mathcal N_1' := \mathcal N_1 \cap \mathcal N'
\ee
(recall $\delta=\frac {\rho_1-\rho}3$).
It follows from \eqref{5.12} that
\be\label{5.14}
\hd(\mathcal N_1\backslash \mathcal N_1')< \lambda^{1-\delta+\epsilon}.
\ee

%\subsection{Bounding $\length \mathcal N_1'$}
%Using the information gathered above, we make the following remark for
%later use.

%We wish to bound the length of $\mathcal N_1'$.

%Recall that we have
%an eigenfunction $\vp$ be an eigenfunction on $\mathbb T^2$ of
%eigenvalue $E=\lambda^2$ and $\mathcal N=\nodal_\vp$ its nodal set.
%Let

Consider a partition of $\mathbb T^2$ in squares $Q_\nu$ of size
$\frac 1\lambda$ and let
$$
\Omega=\bigcup_{Q_\nu\cap\mathcal N'\not=\phi} Q_\nu \subset\mathbb T^2.
$$
We wish to bound the area $|\Omega|$.
\medskip

First, observe that in general for $x, y\in\mathbb T^2$
$$
\big|\psi(x)-\psi(y)\big|\leq \Vert\^\psi\Vert_1
          [\diam(\supp  \^\psi)]|x-y|
$$
and hence
\be\label{4.18}
\sup_{x,y\in Q_\nu } \left| |\psi(x)|-|\psi(y)| \right|
\lesssim |\mathcal E'|^{\frac 12} r\lambda^{-1}
<\lambda^{-\frac{1-\rho_1} 2}.
\ee
It follows that
\begin{equation}
  \sup_\Omega |\psi| < \lambda^{-\delta} +\lambda^{-\frac{1-\rho_1} 2}
\end{equation}
%$|\psi|<\lambda^{-\delta} +\lambda^{-\frac\tau 2}$ on
%$\Omega$.

\subsection{Proof of Theorem~\ref{Theorem2}  ($\rho<\frac 12$).}
Let $\rho=\frac 12-5\delta$, $\rho_1 =\frac 12-2\delta$ for some $\delta>0$,
so that
\be\label{4.19}
r<\lambda^{\frac 12-\delta} %\text { for some } \sigma>0.
\ee
and
$$
\sup_\Omega |\psi| < \lambda^{-\delta}
$$
 From \eqref{4.13} and the preceding %(since $\tau>\frac 12$)
$$
\lambda^{-\delta}  > (c|\Omega|)^{1/\delta}
$$
implying
\be\label{4.20}
|\Omega|< \lambda^{-\delta^2}.
\ee
Thus $\Omega$ contains at most $\lambda^{2-\delta^2}$ boxes $Q_\nu$,
and Lemma \ref{Lemma4.8} implies
\be\label{4.21}
\hd(\mathcal N')=\sum_{Q_\nu\cap \mathcal N' \neq  \emptyset}
 \hd (\mathcal N\cap Q_\nu)\ll \lambda^{1-\delta^2+\ve}< \lambda^{1-\frac {\delta^2}2}
\ee
for $\lambda$ large enough.
%Since $\diam\supp( \^\psi)\ll \lambda^{\rho_1+\ve}<
%\lambda^{\frac 12-\sigma}$,   \eqref{4.21} applies  and
Thus
\be\label{5.15}
\hd(\mathcal N_1')<\lambda^{1-\frac{\delta^2}2}.
\ee
Using \eqref{5.14}, we therefore get
\be\label{5.16}
\hd(\mathcal N_1)< \lambda^{1-\frac\delta 2}
\ee
and since $\hd(\mathcal N_0\backslash\mathcal N_1)<\frac 12 \hd(\mathcal
N_0)$, if we take $c_0<\frac 14$ in \eqref{5.10}, we get
\be\label{5.17}
\hd(\mathcal N_0)< 2\lambda^{1-\frac\delta 2}\;.
\ee
This proves Theorem~\ref{Theorem2}. \qed
\medskip

As pointed out earlier, the validity of Conjecture \ref{Conjecture2}
would allow to replace the restriction $\rho<\frac 12$ by
$\rho<1$.

\subsection{Proof of Theorem~\ref{Theorem3}  ($\rho<1$)}

For general $\rho<1$, write $\rho=1-5\delta$, take
$\rho_1=\rho+3\delta=1-2\delta$, and  apply \eqref{4.17} to bound
$|\Omega|$,  getting
\be
-\delta (\log\lambda)|\Omega|> -\frac{\log\lambda}{\log\log\lambda}
\ee
that is
\be\label{4.22}
|\Omega|<\frac{1/\delta}{\log\log \lambda}
\ee
and hence in this case from \eqref{4.22} and Lemma \ref{Lemma4.9}, we
clearly get
\be\label{4.23}
\hd(\nodal_\vp\backslash \mathcal N') =
\sum_{Q_\nu\cap\Omega=\phi} \hd(\nodal_\vp\cap Q_\nu)> c_1 \lambda
\ee
%According to \eqref{4.23}
%\be
%\label{5.18}
%\hd(\nodal_\vp\backslash\mathcal N_1')>c_1\lambda
%\ee
where $c_1$ is some absolute constant.
Hence, from \eqref{5.10}, \eqref{5.14}, \eqref{4.23}%{5.18}
$$
\begin{aligned}
\hd(\nodal_\vp \backslash\mathcal N_0)&\geq \hd(\nodal_\vp\backslash \mathcal N_1')- \hd (\mathcal N_0\backslash \mathcal N_1) - \hd(\mathcal
 N_1 \backslash \mathcal N_1')\\[6pt]
&> c_1\lambda -c_0 \lambda-\lambda^{1-\delta+\epsilon} >\frac 12 c_1\lambda
\end{aligned}
$$
if we choose $c_0$ small enough. This proves Theorem~\ref{Theorem3}.
\qed

\appendix
\section{Higher order regularity - An example}\label{appendix:example}

The purpose of what follows is to show that `regular arcs' need not
satisfy higher order smoothness bounds, even for $\kappa$ small.
\medskip

Consider the eigenfunction
\be\label{6.1}
\vp(x, y) =\sin (ky+x)+\ve\sin(ky-x)+\delta\sin (y+kx)
\ee
with eigenvalue  $E=1+k^2$, where
$$
\ve=10^{-10} \text { and } \delta=10^{-100} k^{-2}.
$$
We restrict $x\in I=-\frac \pi 4+ [-10^{-3}, 10^{-3}]$ and consider
the curve $C\subset \{\vp=0\}$ parameterized by $y=y(x), x\in I$,
such that \be\label{6.2} |ky(x)+x|= 0(\ve). \ee

Evaluate $y', y'', y'''$
$$
\cos(ky+x)(ky'+1)+\ve\cos(ky-x) (ky'-1)+\delta\cos(y+kx)(y'+k)=0
$$
\be\label{6.3}
y'=\frac{\cos(ky+x)-\ve \cos(ky-x)+\delta k\cos (y+kx)}
{k[\cos(ky+x)+\ve\cos(ky-x)]+\delta\cos(y+kx)}=-\frac 1k +0\Big(\frac\ve k+\delta\Big)
\ee
and
$$
\begin{aligned}
&[k\cos(ky+x)+k\ve\cos(ky-x)+\delta\cos (y+kx)]y''=\\[6pt]
&\sin(ky+x)(ky'+1)^2 +\ve\sin (ky-x)(ky'-1)^2 +\delta\text{ sin}(y+kx)(y'+k)^2=
\end{aligned}
$$
(since $\vp =0$)
$$
\begin{aligned}
&\ve\sin(ky-x)[(ky'-1)^2 -(ky'+1)^2] +\delta\sin(y+kx)[(y'+k)^2-(ky'+1)^2]\\[6pt]
&= -4k\ve\sin (ky-x)y'-\delta(k^2 -1) \sin(y+kx)\big((y')^2 -1\big)
\end{aligned}
$$
and
\begin{equation*}
\begin{aligned}
y''&=\frac{-4k\ve\big(\sin 2x +0 (\ve)\big)\big(\frac 1k+0(\frac\ve
  k)\big)+0(\delta k^2)}
{k\big(1+0(\ve+\frac\delta k)\big)}\\[6pt]
&= -\frac {4\ve}k \sin 2x +0\Big(\frac{\ve^2}k+\delta k\Big)\\[6pt]
%\be\label{6.4}
&\sim \frac{4\ve}k
\end{aligned}
\end{equation*}
from the choice of $\ve$, $\delta$ and $I$.

Thus $C$ is convex with curvature $\sim \frac 1k$.

Next, from the preceding
$$\begin{aligned}
k\big(1+0(\ve)\big)y''' + 0(k\ve|y''|)&= 0(k\ve|y'|+k\ve|y''|+k^2\delta|y'| \, |y''|)+\\[6pt]
&\delta(k^2-1)\big(1-(y')^2\big)(k+y') \cos (y+kx)
\end{aligned}
$$
and
\begin{equation*}
\label{6.5}
\big(1+0(\ve)\big) y'''=\delta(k^2-1) \cos (y+kx)+0\Big(\frac\ve k+\delta\Big) =\delta k^2 \cos kx 
+0\Big(\frac 1k\Big) 
\end{equation*}
where $\delta k^2= 10^{-100}$.
It follows that $\Vert y^{(iv)}\Vert_\infty\sim k$ for large $k$.

\end{document}